\documentclass[a4paper,UKenglish]{amsart}

\usepackage{cancel}
\usepackage{tikz}
\usepackage{centernot}
\usepackage{enumerate}
\usepackage{float}
\usepackage{amsmath}
\usepackage{amsthm}

\usepackage{hyperref}
\usepackage{cleveref}
\newtheorem{theorem}{Theorem}[section]

\newtheorem{corollary}[theorem]{Corollary}
\newtheorem{lemma}[theorem]{Lemma}
\newtheorem{proposition}[theorem]{Proposition}

\theoremstyle{definition}

\newtheorem{definition}[theorem]{Definition}
\newtheorem{example}[theorem]{Example}

\newtheorem{remark}[theorem]{Remark}

\newcommand{\nice}[1]{\text{nice}_{#1}}

\usepackage{amsfonts}

\usepackage{amssymb}
\newcommand{\K}{\mathfrak{K}}

\newcommand{\A}{\mathcal{A}}
\newcommand{\esse}{\mathcal{S}}
\newcommand{\B}{\mathcal{B}}
\newcommand{\La}{\mathcal{L}}
\newcommand{\Mod}{\mathrm{Mod}}
\newcommand{\nats}{\mathbb{N}}

\newcommand{\Lomega}{\mathcal{L}_{\omega_1\omega}}
\newcommand{\equalitybaire}{=_{\nats^\nats}}
\newcommand{\equalitynats}{=_{\nats}}

\newcommand{\permission}[1]{\underset{#1}{\pmb{\downharpoonleft\!\upharpoonright}}}
\newcommand{\ldof}[1]{\mathrm{LD}(#1)}

\newcommand{\ex}{\mathbf{Ex}}

\newcommand{\learnreducible}[1]{\leq_{\mathrm{Learn}}^{#1}}

\newcommand{\Baire}{\nats^\nats}
\newcommand{\learnequiv}[1]{\equiv_{\mathrm{Learn}}^{#1}}

\newcommand{\distinguish}[2]{\underset{#1}{\otimes}(#2)}

\newcommand{\full}[1]{\textit{full}_{#1}}

\newcommand{\initial}[1]{\textit{initial}_{#1}}

\newcommand{\canonicalfull}[1]{\mathsf{F}_{#1}}
\newcommand{\canonicalinitial}[1]{\mathsf{I}_{#1}}

\newcommand{\thsigma}[2]{\mathrm{Th}_{\Sigma_{#1}^{\mathrm{inf}}}(#2)}

\newcommand{\sigmainf}[1]{\Sigma_{#1}^{\mathrm{inf}}}

\newcommand{\piinf}[1]{\Pi_{#1}^{\mathrm{inf}}}

\newcommand{\str}[1]{\langle #1 \rangle}

\DeclareMathOperator*{\bigdoublewedge}{\bigwedge\mkern-15mu\bigwedge}
\DeclareMathOperator*{\bigdoublevee}{\bigvee\mkern-15mu\bigvee}

\title{On the learning power of Friedman-Stanley jumps} 

\author{Vittorio Cipriani}
\address{Institute of Discrete Mathematics and Geometry,
Technische Universit{\"a}t Wien, Wien, Austria}
\email{\href{mailto:vittorio.cipriani17@gmail.com}{vittorio.cipriani17@gmail.com}}

\author{Alberto Marcone}
\address{Dipartimento di Scienze Matematiche, Informatiche e Fisiche,  Universit\`a di Udine, Italy}
\email{\href{mailto:alberto.marcone@uniud.it}{alberto.marcone@uniud.it}}

\author[L.~San Mauro]{Luca San Mauro}
\address{Dipartimento di Ricerca e Innovazione Umanistica\\ Università di Bari \\
Italy}
\email{\href{mailto:luca.sanmauro@gmail.com}{luca.sanmauro@gmail.com}}

\thanks{The first author was supported by the Austrian Science Fund (FWF) 10.55776/P36781. The second author was partially supported by the Italian PRIN 2022 \emph{Models, sets and classifications}, prot.\ 2022TECZJA, funded by the European Union - Next Generation EU. The second and third authors are members of INdAM-GNSAGA}

\keywords{Descriptive set theory,  algorithmic learning theory,  Friedman-Stanley jump, continuous reducibility}
\subjclass[2020]{03E15, 03C57, 68Q32}

\begin{document}
\begin{abstract}
Recently, a surprising connection between algorithmic learning of algebraic structures and descriptive set theory has emerged. 
Following this line of research, we define the learning power of an equivalence relation $E$ on a topological space as
the class of isomorphism relations with countably many equivalence classes that are
continuously reducible to $E$. In this paper, we describe the learning power of the finite Friedman-Stanley jumps of $\equalitynats$ and $\equalitybaire$, proving that these equivalence relations learn the families of countable structures that are pairwise distinguished by suitable infinitary sentences. Our proof techniques introduce new ideas for assessing the continuous complexity of Borel equivalence relations.
\end{abstract}

\maketitle

\section{Introduction}
\label{sec:introduction}
The study of the complexity of isomorphism relations (and, more broadly, natural classification problems) is a major area of research in mathematical logic. Certain classes of structures admit well-behaved systems of invariants, which vastly simplify the task of determining whether given members of the class are isomorphic: a classic example of this phenomenon is the notion of dimension in algebra. However, it has long been recognized that for some natural classes of structures (such as torsion-free abelian groups of rank greater than one) finding suitable invariants can be quite challenging \cite{MR1937205}, while other classes (such as \emph{all} countable groups) are simply too rich to hope for any nice classification. More generally, the complexity of a classification problem is determined by the complexity of the equivalence relation.

Descriptive set theory provides a powerful framework for analyzing classification problems; a crucial tool within this framework is Borel reducibility, which allows comparison of the complexity of equivalence relations in Polish spaces \cite{gao2008invariant}. However, Borel reducibility is a coarse notion as is designed to capture the idea of a function being reasonably constructive (e.g., not relying on the axiom of choice). This coarseness entails significant limitations: every Borel equivalence relation $E$ with countably many  $E$-classes is \emph{smooth} \cite[Proposition 5.4.4]{gao2008invariant}, meaning that $E$ is Borel reducible to  $=_{\nats^\nats}$. As a result, finer notions are required to suitably assess the complexity of isomorphism relations consisting of countably many isomorphism types. This is a well-known issue which motivates research on, e.g., Turing computable embeddings \cite{CCKM04,KMV07} or, more recently, countable reductions \cite{miller2021computable}. 
Here, we explore an alternative framework for ranking those isomorphism relations that escape the standard descriptive set theoretic analysis. Such a framework is centered on continuous reductions and is inspired by an apparently remote area of research: algorithmic learning theory (ALT).




ALT was introduced in the 1960s and has since grown into a broad field encompassing various models of inductive inference \cite{Osh-Sto-Wei:b:86:stl}. ALT has been applied to problems involving learning structural data, e.g., \cite{HaSt07,MS04}, leading to a framework for learning (countable families of nonisomorphic) countable structures. This framework, explored in a recent series of papers \cite{bazhenov2020learning,bazhenov2021turing,bazhenov2022calculating,bazhenov2023learning,rossegger2024learningequivalencerelationspolish}, models a scenario in which a learner, provided with increasing information about a structure $\A$, attempts to guess the isomorphism type of $\A$. The convergence criterion, known as $\mathbf{Ex}$-learning (for \emph{explanatory} learning), requires that the learner stabilizes on the correct conjecture after finitely many stages (for a formal discussion about $\ex$-learning and additional motivation, see \cite{bazhenov2023learning,bazhenov2024classifyingdifferentcriterialearning}). 

It turns out that the $\mathbf{Ex}$-learnability (or, lack thereof) of a family $\K$ of structures has a natural descriptive set theoretic interpretation: \emph{$\K$ is $\mathbf{Ex}$-learnable if and only if the isomorphism relation on $\K$ is continuously reducible to the relation $E_0$  of eventual agreement on the reals} \cite[Theorem 3.1]{bazhenov2023learning}. By replacing $E_0$ with other equivalence relations $E$, one unlocks a learning hierarchy (see \Cref{def:learnreducibility}) that enables to measure how hard is to learn a given family of structures. In fact, we may define the \emph{learning power} of an equivalence relation $E$ on a topological space as the class of isomorphism relations with countably many equivalence classes that are continuously reducible to $E$ --- this provides a way of measuring the classification power of $E$ which is more nuanced than the Borel analysis.


A natural approach for understanding the learning power of a given $E$ is by locating it within the learning hierarchy. However, a deeper understanding is gained by individuating  a syntactic characterization of which families $\K$ of structures are $E$-learnable. This is a nontrivial task which often requires to compute the logical complexity of the sentences needed to  distinguish between members of $\K$. The most natural logic for this task is the infinitary logic $\mathcal{L}_{\omega_1\omega}$. For example, we have that \emph{a family $\K$ of structures is $E_0$-learnable (or, equivalently, $\mathbf{Ex}$-learnable) if and only if each structure in $\K$ is separated by a $\sigmainf{2}$-sentence which holds for it but fails for all other structures in $\K$} \cite[Theorem 3]{bazhenov2020learning}. A few other benchmark equivalence relations, well-studied in the Borel theory, admit similar, though often more intricate, characterizations (see, e.g., \cite[Theorem 5.4]{bazhenov2023learning}).

This paper greatly advances knowledge of the learning power of Borel equivalence relations, by concentrating on the \emph{Friedman-Stanley jump} operator, denoted $E\mapsto E^+$, which is a fundamental yardstick for gauging the complexity of isomorphism problems  and has been a part of the theory of Borel equivalency relations since its inception \cite{FriedmanStanley}. The \emph{Friedman-Stanley tower} is obtained by starting with $=_{\nats^\nats}$ and then iterating the Friedman-Stanley jump transfinitely, along the countable ordinals. Friedman and Stanley  proved that this tower forms a cofinal family of Borel isomorphism relations \cite{FriedmanStanley}.

Our main theorem says that the learning power of the finite levels of the  Friedman-Stanley tower admits a very natural learning theoretic interpretation: 

\begin{theorem}\label{main:thm}
  Let $\K$ be a countable family of pairwise nonisomorphic countable structures  with domain $\nats$. For $n\geq 1$,
  \begin{itemize}
      \item $\K$ is $=^{n+}_{\nats}$-learnable if and only if $\K$ is a $\sigmainf{2n-1}$-poset;
      \item $\K$ is $=^{n+}_{\nats^\nats}$-learnable if and only if $\K$ is a $\sigmainf{2n}$-poset.
  \end{itemize}
\end{theorem}

\noindent (A family $\K$ forms a $\sigmainf{m}$-poset, if the structures from $\K$ can be pairwise distinguished by $\sigmainf{m}$-sentences; see \Cref{definition sigma posets} below).

\smallskip

Such a result gives new insight into the complexity of the so-called \emph{back-and-forth relations}, which play a major role in computable structure theory (see, e.g., \cite{montalban2015robuster, harrison2018scott,MR2289895}). By a classic result of Karp \cite{karp2014finite}, for any given language $\La$, the $n$-th back-and-forth relation corresponds to the relation $\equiv_n$ of $\sigmainf{n}$-equivalence on countable $\La$-structures. Our main theorem, slightly rephrased, provides sharp bounds to the complexity of any restriction of $\equiv_n$ to at most countably many equivalence classes. Specifically,
if $X$ is a collection of $\La$-structures which  intersects countably many $\equiv_n$-classes and is closed under isomorphism, then $\equiv_n\restriction_X$ is continuously reducible to $=^{k+}_\nats$, if $n=2k-1$, or to $=^{k+}_{\nats^\nats}$, if $n=2k$.

We take \Cref{main:thm}  as a solid piece of evidence supporting the naturalness of the
learning theoretic framework for complementing the Borel analysis. Furthermore, we anticipate that our proof techniques (e.g., the permission operators
covered in \Cref{permission operators}) may be used in the future to pursue a more systematic analysis of the continuous complexity of Borel equivalence relations --- a subject that, with only a few notable exceptions \cite{lecomte2020complexity}, has been somewhat overlooked in the literature.

\subsection*{Organization of the paper} In \Cref{sec:preliminaries} and \Cref{sec:techincaltools}, we introduce the preliminaries and  technical tools needed for proving \Cref{main:thm}. Such a proof runs through \Cref{sec:thmforpairs} and \Cref{sec:compactness}:  in \Cref{sec:thmforpairs}, we prove the theorem for pairs of structures; 
 in \Cref{sec:compactness}, we prove a compactness result (\Cref{thm:compactness}) which allows to extend our main theorem to families of structures of any countable size. Finally, \Cref{sec:finalsection} contains some concluding remarks on the learning power of Friedman-Stanley jumps for equivalence relations beyond $\equalitynats$ and $\equalitybaire$. 

\section{Preliminaries} \label{sec:preliminaries}
We assume the reader to be familiar with basic notions of descriptive set theory, infinitary logic, and computable structure theory, as can be found, e.g., in the textbooks \cite{gao2008invariant, marker2016lectures, AK00}. Despite that, we fix the notation and recall the main concepts needed for our proofs.

\subsection{Spaces, languages, structures}
In this paper, we concentrate on the following Polish spaces: $\nats$, $\Baire$ (the Baire space), and products of $\Baire$. For $n>0$, we denote by $\nats^{\nats^n}$ the $n$-th product of $\Baire$.
The concatenation of two finite sequences $\sigma,\tau$ is denoted by $\sigma\tau$. For $n,k \in \nats$, we denote by $n^k$ the sequence made of $k$ many $n$’s; in case $k = 1$ we just write $n$. We use $n^\nats$ to denote the infinite sequence with constant value $n$. Given $p \in \nats^{\nats^n}$  we define the $k$\emph{-th column of} $p$ as $p^{[k]} := p(k,\cdot) \in \nats^{\nats^{n-1}}$; note that, if $n=1$, then $p^{[k]}=p(k) \in \nats$. We fix once for all an effective bijective encoding of pairs, denoted $\str{\cdot, \cdot}:\nats \times \nats \rightarrow \nats$.

For a language $\La$, we denote by $\Mod(\La)$ the space of all $\La$-structures with domain $\nats$. Through this paper we work with a fixed countable language  $\La$ with at least one binary relation symbol\footnote{This choice is to avoid that the corresponding isomorphism relation becomes too simple: indeed, if $\La$ contains only unary relation symbols, then the isomorphism relation on $\Mod(\La)$ Borel reduces to $\equalitybaire^+$  see \cite[Section 13.1]{gao2008invariant}.}. In fact, without loss of generality, we may assume that $\La$ is relational (by replacing any function symbol of $\La$ with its graph).

All the structures we consider belong to $\Mod(\La)$: i.e, they are countable and have domain $\nats$.  Observe that every $\A \in \Mod(\La)$ can be represented as the union of an increasing sequence of its finite
substructures: $\A=\bigcup_{s\in\nats} \A\restriction_s$, where $\A\restriction_s$ denotes the restriction of $\A$ to the set $\{0,\ldots,s-1\}$.

\subsection{Equivalence relations} 
We focus on the following  equivalence relations: $=_{\nats}$ (the identity on the natural numbers), $\equalitybaire$ (the identity on the reals)
and those equivalence relations obtained from the previous ones by finitely iterating the \emph{power operator} or the \emph{Friedman-Stanley jump}. For any equivalence relation $E$ on $X$,
\begin{itemize}
    \item the \emph{power} of $E$ is the equivalence relation $E^\omega$ on $X^\nats$ given by
$$p\ E^\omega\ q : \iff (\forall k)(p^{[k]}\ E\ q^{[k]});$$
\item the \emph{Friedman-Stanley jump} of $E$ is the equivalence relation $E^+$ on $X^\nats$ given by
\[p\ E^{+}\ q : \iff  \{[p^{[k]}]_{E}:k \in \nats\}=\{[q^{[k]}]_{E}: k \in \nats\}.\]
\end{itemize}

\begin{remark}
In the literature,
 $\equalitybaire^+$ is sometimes denoted by $E_{set}$.
\end{remark}


 As aforementioned, continuous reductions represent the most natural way of classifying equivalence relations in our setting:  an equivalence relation $E$ on $X$ is \emph{continuously reducible} to $F$ on $Y$, if there exists a continuous function $f: X\to Y$ such that $
x\, E \, y$ if and only if $f(x)\, F\, f(y)$, for all $x,y\in X$.

Finally, for an equivalence relation $E$, we  denote  by $E^{n(\omega+)}$ the alternate iteration, for $n$ many times, of the power operator and the Friedman-Stanley jump. That is, we define $E^{0(\omega+)}:= E$ and $ E^{(n+1)(\omega+)}:=((E^{n(\omega+)})^\omega)^+$.

\subsection{Infinitary logic}
We refer the reader to  \cite{marker2016lectures} for background and notation on infinitary logic. We denote by $\mathcal{L}_{\omega_1\omega}$ the infinitary language which allows conjunctions and disjunctions of countably infinite sets of formulas.

\begin{definition}
For every ordinal $\alpha<\omega_1$ we define the sets 
$\sigmainf{\alpha}$ and $\piinf{\alpha}$ of $\mathcal{L}$-formulas inductively. For $\alpha=0$, the $\sigmainf{\alpha}$ and $\piinf{\alpha}$-formulas are the quantifier-free first-order $\La$-formulas.
\begin{itemize}
   \item A \emph{$\sigmainf{\alpha}$-formula} $\varphi(\overline{x})$ is a countable disjunction
  \[ \bigdoublevee_{i \in I} (\exists \overline{y}_i)( \psi_i(\overline{x},\overline{y}_i))\]
 where $I$ is a countable set and each $\psi_i$ is a $\piinf{\beta_i}$-formula for some $\beta_i < \alpha$;
\item A \emph{$\piinf{\alpha}$-formula} $\varphi(\overline{x})$ is a countable conjunction
\[ \bigdoublewedge_{i \in I} (\forall \overline{y}_i)( \psi_i(\overline{x},\overline{y}_i))\]
where $I$ is a countable set and each $\psi_i$ is a $\sigmainf{\beta_i}$-formula for some $\beta_i < \alpha$.
\end{itemize}
The  $\sigmainf{n}$\emph{-theory} of a structure $\esse$, denoted by $\thsigma{n}{\esse}$, is the collection of $\sigmainf{n}$-sentences true of $\esse$.
\end{definition}

\begin{remark}
    \label{remark:nsigman}
    Notice that, since all our $\mathcal{L}$-structures have domain $\nats$, existential (respectively, universal) quantifiers can be replaced by countable disjunctions (conjunctions) in $\mathcal{L}(\nats)$, the expanded language obtained adding to $\mathcal{L}$ a constant symbol for each natural number. Hence, we can assume that our infinitary formulas are quantifier-free. In particular, $\thsigma{n}{\A} \neq  \thsigma{n}{\B}$, for two $\mathcal{L}$-structures $\A$ and $\B$, is always witnessed by a $\sigmainf{n}$ quantifier-free sentence in $\mathcal{L}(\nats)$.

\end{remark}

\subsection{Classifying families of structures}
\label{sec:learnability}
To classify the complexity of classes of structures it is customary to study different notions of \emph{reductions} between them, e.g., \emph{Borel reducibility} \cite{FriedmanStanley} in descriptive set theory and \emph{Turing computable embeddings} \cite{CCKM04,KMV07} in computable structure theory. In this paper we are interested in \emph{continuous embeddings}, which are used in \Cref{sec:righttoleft} to prove the left-to-right direction of \Cref{main:thm}.

\begin{definition}
Let $\mathfrak{L}_0$ and $\mathfrak{L}_1$ be two collections of structures closed under isomorphism. A \emph{continuous embedding} from $\mathfrak{L}_0$ into $\mathfrak{L}_1$ is a continuous function $\Gamma: \mathfrak{L}_0\rightarrow \mathfrak{L}_1$ such that, 
          for every $\A,\B \in \mathfrak{L}_0$, $\A \cong \B \iff \Gamma(\A) \cong \Gamma(\B)$.
\end{definition}

The following theorem is a consequence of the relativized version of the well-known Pullback Lemma from \cite{KMV07}, see \cite{bazhenov2020learning}.

\begin{lemma}
\label{lem:relpull}
Let $\Gamma$ be a continuous embedding from $\mathfrak{L}_0$ to $\mathfrak{L}_1$. Then, for any infinitary sentence $\varphi$ in the language of $\mathfrak{L}_1$  one can find an infinitary sentence $\varphi^*$ in the language of $\mathfrak{L}_0$ such that for all $\A \in \mathfrak{L}_0$, we have $\A \models \varphi \iff \Gamma(\A)\models \varphi^*$.
    Furthermore, for any non-zero $\alpha < \omega_1$, if $\varphi$ is a $\Sigma_\alpha^{\mathrm{inf}}$-sentence, we can assume that $\varphi^*$ is $\Sigma_\alpha^{\mathrm{inf}}$ as well.
\end{lemma}

 As aforementioned, in the setting of learning we are interested in \emph{countable families of pairwise nonisomorphic countable structures}; henceforth, we refer to such a family simply as a \lq\lq family of structures\rq\rq. 

 \begin{definition}\label{definition sigma posets}
A family of structures $\K=\{\A_i : i \in \nats\}$ is a \emph{$\sigmainf{n}$-poset}, if
		$\thsigma{n}{\A_i} \neq  \thsigma{n}{\A_j}$, for all $i\neq j\in\nats$.
	
\end{definition}

However, to define $E$-learnability we need to close our families of structures under isomorphism and we adopt a classical terminology of ALT.

  \begin{definition}\label{learning domain}
  For a family $\K=\{\A_i:i\in\nats\}$ of structures, the \emph{learning
domain} $\ldof{\K}$  is the set of all isomorphic copies of the structures in  $\K$, i.e.,  $\ldof{\K}: =\bigcup_{i\in\nats} \{\esse \in \Mod(\La): \esse\cong \A_i\}$. (In descriptive set theoretic terminology, $\ldof{\K}$ is the $\cong$-saturation of $\K$.)
  \end{definition}

\begin{definition}\cite[Definition 3.2]{bazhenov2023learning}
    Let $E$ be an equivalence relation on a topological space $X$. A family $\K$ of structures is \emph{$E$-learnable} if and only if there is a continuous function $\Gamma:\ldof{\K}\to X$ so that
    \[
    \A \cong \B \iff \Gamma(\A) \, E\, \Gamma(\B),
    \]
    for all $\A,\B\in \ldof{\K}$. If such a $\Gamma$ exists, we may abuse terminology and say  that the family  $\K$ is continuously reducible to $E$.
\end{definition}

As discussed in the introduction, the learning theoretic terminology of the last definition is justified by the fact that $E_0$-learning coincides the classic learning criterion called $\ex$-learning. 

\begin{definition}{\cite[Definition 3.3]{bazhenov2023learning}}
\label{def:learnreducibility}
An equivalence relation $E$ is \emph{learn-reducible} to an equivalence relation $F$ (in symbols, $E \learnreducible{} F$) if every $E$-learnable family is also $F$-learnable.
\end{definition}

\section{Some technical tools}
\label{sec:techincaltools}
We introduce the main technical ingredients required for \Cref{sec:thmforpairs}.
\subsection{Full and initial}
\label{sec:fullinitial}

We begin by defining suitable elements, named \emph{canonical} $\full{n}$ (denoted by $\canonicalfull{n}$) and \emph{canonical} $\initial{n}$ (denoted by $\canonicalinitial{n}$), which we use to continuously classify pairs of structures distinguished by $\sigmainf{n}$-sentences.

  For all $n\geq 1$, $\canonicalfull{n},\canonicalinitial{n}$ belong to $\nats^{\nats^n}$ and are defined  by induction, with both an odd ($\canonicalfull{1},\canonicalinitial{1}$) and an even ($\canonicalfull{2}, \canonicalinitial{2}$) base case: 
\begin{itemize}
\item $\canonicalfull{1}=01^\nats \text{ and }  \canonicalfull{2} \text{ is such that, for any } k$,
\[
\ \canonicalfull{2}^{[k]}=\begin{cases}
			0^\nats & \text{if } k=0\\
			0^{k-1}1^\nats & \text{otherwise;}
		\end{cases}
\]
\item $\canonicalinitial{1}=1^\nats \text{ and }\canonicalinitial{2} \text{ is such that, for any } k, \ \canonicalinitial{2}^{[k]}=
			0^{k}1^\nats.$
\end{itemize}

		For $n >2$, 
  $\canonicalfull{n}$ is such that for any $i,k$, 
\begin{itemize}
    \item ${\canonicalfull{n}^{[0]}}^{[k]}= \canonicalfull{n-2} \text{ and } {\canonicalfull{n}^{[i+1]}}^{[k]}= 
	  \begin{cases}
	  \canonicalfull{n-2}& \text{if } k < i\\
	  \canonicalinitial{n-2} & \text{otherwise;}
	  \end{cases}$
\item ${\canonicalinitial{n}^{[i]}}^{[k]}=
	  \begin{cases}
	  \canonicalfull{n-2}& \text{if } k < i\\
	  \canonicalinitial{n-2} & \text{otherwise.}
	  \end{cases}$
\end{itemize}

	So, $\canonicalinitial{n}$ is the same as $\canonicalfull{n}$ except for the fact that it is missing the first column (when $n=1$, such column is actually a coordinate). Thus, we call the \lq\lq special\rq\rq\ column $\canonicalfull{n}^{[0]}$ the \emph{canonical $\full{n}$-witness}, while we call $\canonicalfull{n}^{[i+1]}=\canonicalinitial{n}^{[i]}$ the \emph{canonical $\initial{n}^{i}$-witness}.	

Notice that, for $n \in \nats$, $\canonicalfull{2n+1}$ and $\canonicalinitial{2n+1}$ are not $(\equalitynats^+)^{n(\omega+)}$-equivalent, while $\canonicalfull{2n+2}$ and $\canonicalinitial{2n+2}$ are not  $(\equalitybaire^+)^{n(\omega+)}$-equivalent. These equivalence classes are our invariants. It is convenient to use the following terminology:
\begin{itemize}
\item an element $p \in \nats^{\nats^{2n+1}}$ is $\full{2n+1}$ if $p \ (\equalitynats^+)^{n(\omega+)} \ \canonicalfull{2n+1}$, and $\initial{2n+1}$ if $p \ (\equalitynats^+)^{n(\omega+)} \ \canonicalinitial{2n+1}$;
\item an element $p \in \nats^{\nats^{2n+2}}$ is $\full{2n+2}$ if $p \ (\equalitybaire^+)^{n(\omega+)} \ \canonicalfull{2n+2}$, and $\initial{2n+2}$ if $p \ (\equalitybaire^+)^{n(\omega+)} \ \canonicalinitial{2n+2}$.
     \end{itemize}

The following terminology --- exploiting the fact that $(=_X^+)^{n(\omega+)}$ coincides with $((=_X)^{n(+\omega)})^+$ for $X \in \{\nats,\Baire\}$ --- is used to refer to a specific column of a $\full{n}$ or $\initial{n}$ element.
For $n,i \in \nats$, 
\begin{itemize}
\item an element $p \in \nats^{\nats^{2n}}$ is a \emph{$\full{2n+1}$-witness} if $p \ (\equalitynats)^{n(+\omega)} \ \canonicalfull{2n+1}^{[0]}$, and an \emph{$\initial{2n+1}^{i}$-witness} if $p \ (\equalitynats)^{n(+\omega)} \ \canonicalinitial{2n+1}^{[i]}$.
\item an element $p \in \nats^{\nats^{2n+1}}$ is a \emph{$\full{2n+2}$-witness} if $p \ (\equalitybaire)^{n(+\omega)} \ \canonicalfull{2n+2}^{[0]}$ and an \emph{$\initial{2n+2}^{i}$-witness} if  $p \ (\equalitybaire)^{n(+\omega)} \ \canonicalinitial{2n+2}^{[i]}$.
\end{itemize}

\subsection{The $\nice{n,\varphi}$ functions}
\label{sec:nicefunctions}
 The key feature of $\full{n}$ and $\initial{n}$ is that they allow to prove the existence of $\nice{n,\varphi}$ \emph{functions}. 
 \begin{definition}
	\label{def:nicefunction}
Let $n \geq 1$ and $\varphi$ be an $\La$-sentence of $\mathcal{L}_{\omega_1\omega}$. A continuous function $\Gamma : \Mod(\La) \rightarrow \nats^{\nats^n}$ is \emph{$\nice{n,\varphi}$}  if for every structure $\esse$ the  following holds:
         \item 
         \begin{itemize}
             \item if $\esse \models \varphi$ then $\Gamma(\esse)$ is $\full{n}$;
             \item if $\esse \not\models \varphi$ then $\Gamma(\esse)$ is $\initial{n}$.
         \end{itemize}
 \end{definition}

The goal of the next two subsections is to prove the following theorem, from which the right-to-left direction of \Cref{main:thm} for pairs of structures follows easily.

\begin{theorem}
	\label{thm:nicefunction} 
Let $n \geq 1$. For every $\sigmainf{n}$-sentence $\varphi$ there is a $\nice{n,\varphi}$ function. 
\end{theorem}
 The next lemma proves the base cases of \Cref{thm:nicefunction}.
 
 \begin{lemma}	\label{lem:basecases}
	Let $n \in \{1,2\}$. For every $\sigmainf{n}$-sentence $\varphi$ there is a $\nice{n,\varphi}$ function. 
 \end{lemma}
 \begin{proof}
	For both $n$'s the strategy is the same: namely, we define a continuous function $\Gamma$ that monitors whether the input structure satisfies $\varphi$.

	For $n=1$, by \Cref{remark:nsigman}, we can assume that $\varphi:= \bigdoublevee_{i \in I} \psi_i$ where $I$ is a countable set and $\psi_i$ is a quantifier-free $\Lomega$-sentence. Let $\Gamma:\Mod(\La)\rightarrow \Baire$ be the continuous function defined by
	 $$\Gamma(\esse)(2i):=\begin{cases}
		 1 & \text{if } \esse\not\models \psi_i\\
		 0 & \text{otherwise,}
	 \end{cases}$$
	 and $\Gamma(\esse)(2i+1)=1$. It is clear that if $\esse \models \varphi$ then $\Gamma(\esse)$ is $\full{1}$, while if $\esse\not\models\varphi$ then $\Gamma(\esse)$ is  $\initial{1}$. This concludes the proof for $n=1$.
     
     For $n=2$, by \Cref{remark:nsigman} we can assume  $\varphi := \bigdoublevee_{i \in I}\bigdoublewedge_{j \in J}  \psi_{i,j}$ where $I,J$ are countable sets and $\psi_{i,j}$ is a quantifier-free $\Lomega$-sentence. Let $\Gamma:\Mod(\La)\rightarrow {\Baire}^\nats$ be the continuous function defined by

$$\Gamma(\esse)^{[2i]}(j):=\begin{cases}
    0 & (\forall j' \leq j)(\esse \models \psi_{i,j'})\\
    1 & \text{otherwise,}
\end{cases}$$
		and $\Gamma(\esse)^{[2i+1]} = 0^i1^\nats$. Notice that it is always the case that $ \{ 0^n1^\nats:n \in \nats\} \subseteq \{\Gamma(\esse)^{[i]} : i \in \nats\}$.	Furthermore, if $\esse \models \varphi$, then there exists an $i$ such that $\esse \models \psi_{i,j}$ holds for every $j$, and, by the definition of $\Gamma$, we have:
  \begin{itemize}
      \item if $\esse \models \varphi$ then $\{\Gamma(\esse)^{[i]} : i \in \nats\}= \{ 0^n1^\nats:n \in \nats\} \cup \{0^\nats\}$. That is, $\Gamma(\esse)$ is $\full{2}$;
      \item if $\esse \not\models \varphi$ then $\{\Gamma(\esse)^{[i]} : i \in \nats\}= \{ 0^n1^\nats:n \in \nats\}$. That is, $\Gamma(\esse)$ is $\initial{2}$. \qedhere

  \end{itemize}
 \end{proof}

\subsection{Permission operators}\label{permission operators}
To prove the existence of $\nice{n,\varphi}$ functions for $n \geq 3$ we need a new ingredient, namely a \emph{permission operator} that guarantees that all columns of $\Gamma(\esse)$ are either $\initial{n}$ or $\full{n}$, something that was achieved without much effort in \Cref{lem:basecases}.	
These operators are defined for elements of $\nats^{\nats^{n}}$ in general, but we are going to apply them only to elements that are either $\full{n}$ or $\initial{n}$.

First, we define:
\begin{itemize}
\item 
$\permission{-1}(i,j) = \max\{i,j\}$ for all $i,j \in \nats$;
\item 
$\permission{0}(p, q) = \begin{cases}
	p & \text{if } p=q\\
	p\restriction(\ell(p,q))1^\nats & \text{otherwise,}
\end{cases}$\\
where $\ell(p,q)$ is the length of agreement between $p$ and $q$ and $p,q \in \Baire$.
\end{itemize}
Then, for $n\geq 1$, $\permission{n}: \nats^{\nats^{n}} \times \nats^{\nats^{n}} \rightarrow \nats^{\nats^{n}}$ is inductively defined as 
\begin{itemize}
\item 
${\permission{1}(p,q)^{[\str{i,j}]}}=\, \permission{-1}(p^{[i]},q^{[j]})= \max\{p(i),q(j)\}$;
\item ${\permission{2}(p,q)^{[\str{i,j}]}}=\, \permission{0}(p^{[i]},q^{[j]})$;
\item ${\permission{n+2}(p,q)^{[\str{i,j}]}}^{[k]}=\,\permission{n}({p^{[i]}}^{[k]} ,{q^{[j]}}^{[k]} ).$
\end{itemize}

\begin{theorem}
	\label{thm:fullinitial}
	Let $p,q \in \nats^{\nats^{n}}$ and $n\geq 1$:
	
	\begin{itemize}
		\item[(i)] if $p,q$ are $\full{n}$, then $\permission{n}(p,q)$ is $\full{n}$
		\item[(ii)] if at least one between $p$ and $q$ is $\initial{n}$, then $\permission{n}(p,q)$ is $\initial{n}$.
	\end{itemize}
\end{theorem}
\begin{proof}
The proof is by simultaneous induction and has two base cases, namely $n=1$ and $n=2$: we omit their proofs being trivial. Moreover, we only prove item (i), as the proof of item (ii) is easier: indeed, in item (i) we have to take care of the behavior of the permission operators both on $\full{n}$-witnesses and $\initial{n}^i$-witnesses, while in (ii) just on the $\initial{n}^i$-witnesses.

We now prove the inductive step; it suffices to prove the one starting from $2n$, as the one starting from $2n+1$ is the same replacing $\equalitybaire$ with $\equalitynats$. So, assume the claim holds for $2n$. Given two $\full{2n+2}$ elements $p,q \in \nats^{\nats^{2n+2}}$ we want to show that $\permission{2n+2}(p,q)$ is $\full{2n+2}$, i.e., 
	\[\permission{2n+2}(p,q) \ {(\equalitybaire^+)}^{n(\omega+)}\  \canonicalfull{2n+2}.\] 
  This means showing that:
	\begin{itemize}
		\item[(a)] for every $k$, there exists $l$ such that for every $m$ $${\permission{2n}(p,q)^{[l]}}^{[m]}  (\equalitybaire^+)^{(n-1)(\omega+)} {\canonicalfull{2n}^{[k]^{[m]}}}.$$  
	
		\item[(b)] for every $l$, there exists $k$ such that for every $m$ $${\permission{2n}(p,q)^{[l]}}^{[m]}\ (\equalitybaire^+)^{(n-1)(\omega+)}\ 
		{\canonicalfull{2n}^{[k]^{[m]}}}. $$ 
	\end{itemize}

By definition of $\full{2n+2}$, $p$ and $q$ satisfy the following two properties:
	\begin{itemize}
		\item  $p$ and $q$ have a $\full{2n+2}$-witness, i.e., there exist $k_f$ and $l_f$ such that, for every $m$, ${p^{[k_f]}}^{[m]}$ and ${q^{[l_f]}}^{[m]}$  are $\full{2n}$;
		\item  for every $i$, $p$, and $q$ have an $\initial{2n+2}^i$-witness, i.e., there exist $k_i$ and $l_i$ such that, for every $m<i$, ${p^{[k_i]}}^{[m]}$ and ${q^{[l_i]}}^{[m]}$  are $\full{2n}$ while for every $m \geq i$ ${p^{[k_i]}}^{[m]}$ and ${q^{[l_i]}}^{[m]}$  are $\initial{2n}$.
	\end{itemize}
	By induction hypothesis, we have that:
	\begin{itemize}
		\item  $\permission{2n}({p^{[k_f]}}^{[m]},{q^{[l_f]}}^{[m]})$ is $\full{2n}$;
		\item for every $m<i$ $\permission{2n}({p^{[k_i]}}^{[m]},{q^{[l_i]}}^{[m]})$ is $\full{2n}$;
		\item for every $m\geq i$ $\permission{2n}({p^{[k_i]}}^{[m]},{q^{[l_i]}}^{[m]})$ is $\initial{2n}$.
	\end{itemize}

	 Hence, we have that
	 \begin{itemize}
	 	\item ${\permission{2n+2}(p,q)^{[\str{l_f,k_f}]}}^{[m]}=\permission{2n}({p^{[l_f]}}^{[m]} , {p^{[k_f]}}^{[m]} ) $ is $\full{2n}$;
	 		\item for every $m<i$ ${\permission{2n+2}(p,q)^{[\str{l_i,k_i}]}}^{[m]}=\permission{2n}({p^{[l_i]}}^{[m]} , {p^{[k_i]}}^{[m]} ) $ is $\full{2n}$;
	 		 \item for every $m\geq i$ ${\permission{2n+2}(p,q)^{[\str{l_i,k_i}]}}^{[m]}=\permission{2n}({p^{[l'_i]}}^{[m]} , {p^{[k_i]}}^{[m]} ) $ is $\initial{2n}$.
	 \end{itemize}
	 This easily yields (a).

 To prove (b), consider $\permission{2n+2}(p,q)^{[\str{k,l}]}$. The proof of (a) already shows that we do not have to worry when $k=k_f$ and $l=l_f$ (or vice versa)  and when $k=k_i$ and $l=l_i$ (or vice versa). The only remaining cases are: 
 \begin{itemize}
 	\item $p^{[k]}$ is a $\full{2n+2}$-witness and $q^{[l]}$ is an $\initial{2n+2}^i$-witness, for some $i \in \nats$ (or vice versa);

 	\item $p^{[k]}$ and $q^{[l]}$ are respectively $\initial{2n+2}^i$-witness and $\initial{2n+2}^j$-witness, for $i \neq j$.
 		
 \end{itemize}
 We claim that in the first case we obtain that $\permission{2n+2}(p,q)^{[\str{k,l}]}$ is an $\initial{2n+2}^i$-witness while in the second one that $\permission{2n+2}(p,q)^{[\str{k,l}]}$ is an $\initial{2n+2}^{\min\{i,j\}}$-witness. 
 
In the first case, by induction hypothesis, we have that:  
 \begin{itemize}
 	\item for every $m < i$ 
 	 	${\permission{2n+2}(p,q)^{[\str{l,k}]}}^{[m]}=\permission{2n}({p^{[l]}}^{[m]} , {p^{[k]}}^{[m]} ) $ is $\full{2n}$
 		\item for every $m \geq i$ 
 	 	${\permission{2n+2}(p,q)^{[\str{l,k}]}}^{[m]}=\permission{2n}({p^{[l]}}^{[m]} , {p^{[k]}}^{[m]} ) $ is $\initial{2n}$
 \end{itemize}
 i.e., $\permission{2n+2}(p,q)^{[\str{k,l}]}$ is an $\initial{2n+2}^i$-witness. 

 We omit the proof of the second case as it follows the same pattern. This concludes the proof of (b) and the proof of the theorem.
\end{proof}

\section{Proving \Cref{main:thm} for pairs of structures}
\label{sec:thmforpairs}

	\subsection{The right-to-left direction}
 We start by proving \Cref{thm:nicefunction}.

 \begin{proof}[Proof of \Cref{thm:nicefunction}]
     The proof is by induction, and the base cases, namely $n=1$ and $n=2$, are \Cref{lem:basecases}.

  Assume that the statement holds for $n$: we now show it holds for $n+2$ as well. Without loss of generality, by \Cref{remark:nsigman}, let \[\varphi:=\bigdoublevee_{i \in I} \bigdoublewedge_{j \in J} \psi_{i,j}\] where $I,J$ are countable sets and each $\psi_{i,j}$ is a quantifier-free $\sigmainf{n}$-sentence of $\Lomega$.
  
	By the induction hypothesis, for every $i$ and $j$ there exists a $\nice{n,\psi_{i,j}}$ function $\Gamma_{i,j}: \Mod(\La)\rightarrow {\Baire}^{n}$. We proceed as follows.

\begin{enumerate}
	\item First, we compute $q \in \nats^{\nats^{n+2}}$ just letting ${q^{[i]}}^{[j]}= \Gamma_{i,j}(\esse)$ for every $i,j \in \nats$.
	\item Given $q$, we compute an element  $r \in \nats^{\nats^{n}}$ letting, for $i,j \in \nats$,
		\[{r^{[i]}}^{[0]}={q^{[i]}}^{[0]}\text{  and  }{r^{[i]}}^{[j+1]}= \permission{n}({q^{[i]}}^{[j+1]},{r^{[i]}}^{[j]}).\] 
\end{enumerate}
\textit{Claim:} The following hold:
	\begin{itemize}
		\item[(i)] 	if $\esse \models \varphi$, then there exists some $k$ such that $r^{[k]}$ is a $\full{n+2}$-witness and, for every $i$,  $r^{[i]}$ is either a $\full{n+2}$-witness or an $\initial{n+2}^j$-witness for some $j\in \nats$;
		\item[(ii)] if $\esse \not\models \varphi$,  for every $i$,  $r^{[i]}$ is an $\initial{n+2}^j$-witness for some $j\in \nats$.
	\end{itemize}
\begin{proof}[Proof of claim]
	First notice that since every $\Gamma_{i,j}$ is $\nice{n,\psi_{i,j}}$, ${q^{[i]}}^{[j]}=\Gamma_{i,j}(\esse)$ is either $\full{n}$ or $\initial{n}$. 

To prove (i), suppose that $\esse \models \varphi$ and fix $i_0$ such that $\psi_{i_0,j}$ holds for every $j \in \nats$. Then we have that  ${q^{[i_0]}}^{[j]}=\Gamma_{i_0,j}(\esse)$ is $\full{n}$ for every $j \in \nats$. 
By definition ${r^{[i_0]}}^{[0]}$ is $\full{n}$ as well and by \Cref{thm:fullinitial} it is immediate that the same is true for ${r^{[i_0]}}^{[j+1]}= \permission{n}({q^{[i_0]}}^{[j+1]},{r^{[i_0]}}^{[j]})$. This shows that $r^{[i_0]}$ is a $\full{n+2}$-witness.

Now take $i \neq i_0$. If $\psi_{i,j}$ holds for every $j$ then we can apply the same reasoning as above. Otherwise, suppose that there is some $j$ such that $\psi_{i,j}$ does not hold and let $j_1$ be the minimum such $j$. Since every $\Gamma_{i,j}$ is $\nice{n,\varphi}$, we have that  for $j<j_1$ ${q^{[i]}}^{[j]}=\Gamma_{i,j}(\esse)$ is $\full{n}$ and 
${q^{[i]}}^{[j_1]}=\Gamma_{i,j_1}(\esse)$ is $\initial{n}$. Using \Cref{thm:fullinitial} it is easy to notice that for every $j> j_1$,  ${r^{[i]}}^{[j]}= \permission{n}({q^{[i]}}^{[j]},{r^{[i]}}^{[j-1]})$ is $\initial{n}$. Thus $r^{[i]}$ is an $\initial{n+2}^{j_1}$-witness. This concludes the proof of (i).

The proof of (ii) follows the same pattern: clearly, if $\esse \not\models \varphi$ there is no $i_0$ such that $\psi_{i_0,j}$ holds for every $j \in \nats$, and hence a similar reasoning to the one showing that  $r^{[i]}$ is an $\initial{n+2}^{j_1}$-witness proves the claim.
\end{proof}
 The desired $\nice{n+2,\varphi}$ function $\Gamma$ is defined letting 
 \[\Gamma(\esse)^{[2i]}=r^{[i]}\text{ and }\Gamma(\esse)^{[2i+1]} = r'_i\] where $r'_i$ is an $\initial{n+2}^i$-witness. Notice that, regardless of whether $\esse \models \varphi$, $\Gamma(\esse)$ contains in the odd columns all $\initial{n}^i$-witnesses. Then,
\begin{itemize}
	\item if $\esse \models \varphi$, by Claim (i), $\{\Gamma(\esse)^{[2i]}: i \in \nats\}$ contains (at least) a $\full{n+2}$-witness plus $\initial{n}^i$-witnesses. That is, $\Gamma(\esse)$ is $\full{n+2}$;
		\item if $\esse \not\models \varphi$, by Claim (ii), $\{\Gamma(\esse)^{[2i]}: i \in \nats\}$ contains some $\initial{n}^i$-witnesses and no $\full{n+2}$-witness. That is, $\Gamma(\esse)$ is $\initial{n+2}$.
\end{itemize}
	
This concludes the proof of \Cref{thm:nicefunction}.
  \end{proof}
  \begin{corollary}
  \label{corollary:omegaplusstill}
Let $\K$ be a pair of structures. For $n \geq 1$,
\begin{itemize}
\item if $\K$ is a $\sigmainf{2n-1}$-poset, then $\K$ is $(\equalitynats^+)^{(n-1)(\omega+)}$-learnable;
\item if $\K$ is a $\sigmainf{2n}$-poset, then $\K$ is  $\equalitybaire^{n(\omega+)}$-learnable.
\end{itemize}
	 \end{corollary}
  \begin{proof}
           Given a family of structures $\{\A,\B\}$ such that $\thsigma{n}{\A} \neq \thsigma{n}{\B}$, let $\varphi$ be a $\sigmainf{n}$-sentence such that $\A \models \varphi$ and $\B \not\models\varphi$ (or vice versa). To show that $\{\A,\B\}$ is $E^{n(\omega+)}$-learnable (where $E$ is either $\equalitynats^+$ or $\equalitybaire$, depending on the parity of $n$), we need to continuously reduce $\ldof{\{\A,\B\}}$ to $E^{n(\omega+)}$: to do so, it suffices to take a $\nice{n,\varphi}$ function.  Indeed, such a function partitions the $\La$-structures into two $E^{\omega+}$ equivalence classes (depending on whether the structure satisfies $\varphi$) and so, in particular, also partitions the structures in $\ldof{ \{\A,\B\} }$.
  \end{proof}

\subsection{Friedman-Stanley jumps absorb the power operator}
Notice that \Cref{corollary:omegaplusstill} is \lq\lq almost\rq\rq\ the right-to-left direction of \Cref{main:thm} for pairs of structures. This subsection introduces a simple operator which allows to eliminate the power operator $\omega$ and obtain the result we wished for.

Such an operator is $\underset{n,i}{\otimes}: \nats^{\nats^n} \rightarrow \nats^{\nats^n}$ and is defined so that, for every $\vec{x} \in \nats^n$, $\distinguish{n,i}{p}(\vec{x})=\str{p(\vec{x}),i}$.
The operator $\underset{n,i}{\otimes}$ has the following straightforward property.

\begin{proposition}
\label{prop:otimesni}

Let $X \in \{\nats,\Baire\}$. For every $n\geq 1$, and for every $p,q \in \nats^{\nats^n}$ the following hold:
\begin{itemize}
	\item[(a)] for every $i$, $p =_X^{n+} q \iff \distinguish{n,i}{p} =_X^{n+} \distinguish{n,i}{q}$;
	\item[(b)] for every $i \neq j$, $\distinguish{n,i}{p} \neq_X^{n+}\distinguish{n,j}{q}$.
\end{itemize}
\end{proposition}

\begin{proposition}
	\label{prop:otimesproperty}
	Let $X \in \{\nats,\Baire\}$ and $n \geq 1$. Then
	 $(=_X^{n+})^{\omega}$ and $=_X^{n+}$ are continuously bi-reducible.
\end{proposition}
\begin{proof}
	The  continuous reduction of $=_X^{n+}$ to $(=_X^{n+})^{\omega}$ is trivial. For the converse reduction, given $p \in X^{\nats^{n+1}}$, let $\tilde{p} \in X^{\nats^n}$ be such that  $\tilde{p}^{[\str{i,k}]} := \distinguish{n,i}{{p^{[i]}}^{[k]}}$. Then,
    
	\begin{equation*}
		\begin{aligned}
			p\ (=_X^{n+})^\omega\ q 	& \iff (\forall i) \big( 
			(\forall k)(\exists l)({p^{[i]}}^{[k]} =_X^{(n-1)+} {q^{[i]}}^{[l]}) \,\land & \text{by def.\ of $^\omega$ and $^+$}\\
			&\ \ \ \ \ \ \ \ \ (\forall l)(\exists k)({p^{[i]}}^{[k]}  =_X^{(n-1)+} {q^{[i]}}^{[l]}) 
			\big).\\
				&\iff (\forall i) \big( 
			(\forall k)(\exists l)(\distinguish{n-1,i}{{p^{[i]}}^{[k]}} =_X^{(n-1)+}  \distinguish{n-1,i}{{q^{[i]}}^{[l]}} \,\land & \text{by \ref{prop:otimesni}(a)}
			\\
			&\ \ \ \ \ \ \ \ \ (\forall l)(\exists k)(\distinguish{n-1,i}{{p^{[i]}}^{[k]}}  =_X^{(n-1)+}  \distinguish{n-1,i}{{q^{[i]}}^{[l]}})	\big) 
		\\
			&\iff (\forall i)
		\big(
		(\forall k)(\exists l) (\tilde{p}^{[\str{i,k}]} =_X^{(n-1)+} \tilde{q}^{[\str{i,l}]}) \,\land & \text{by \ref{prop:otimesni}(b)}\\
		&\ \ \ \ \ \ \ \ \ (\forall l)(\exists k)(\tilde{p}^{[\str{i,k}]} =_X^{(n-1)+} \tilde{q}^{[\str{i,l}]})
		\big)\\
			& \iff \tilde{p} =_X^{n+} \tilde{q}.
		\end{aligned}\qedhere
	\end{equation*}
\end{proof}

As announced, the right-to-left direction of \Cref{main:thm} restricted to families of size $2$ is an immediate consequence of \Cref{corollary:omegaplusstill} and \Cref{prop:otimesproperty}.

	\begin{corollary}
		\label{cor:lrdirection}
        Let $\K$ be a pair of structures. For $n \geq 1$,
\begin{itemize}
\item if $\K$ is a $\sigmainf{2n-1}$-poset, then $\K$ is $\equalitynats^{n+}$-learnable;
\item if $\K$ is a $\sigmainf{2n}$-poset, then $\K$ is  $\equalitybaire^{n+}$-learnable. 
\end{itemize}
	\end{corollary}

	\subsection{The left-to-right direction}
 \label{sec:righttoleft}
This part of the proof combines well-known results from descriptive set theory and computable structure theory.


 \begin{theorem}[{\cite{karp2014finite}, \cite[Proposition 15.1, Theorem 18.6]{AK00}}]
	\label{thm:hardtodistinguish}
     For $n \geq 1$, the following are
equivalent.
\begin{enumerate}
\item $\thsigma{n}{\B}\subseteq \thsigma{n}{\A}$;
    \item Given $\esse \in \ldof{ \{\A,\B\} }$, deciding whether $\esse \cong \A$ or $\esse \cong \B$ is $\Sigma_n^0$-hard. That is, for every $\Sigma_n^0$ subset $C \subseteq \Baire$ there is a continuous operator $\Phi_C$ such that, $\Phi_C(p) \cong \A$ if $p \in C$, and $\Phi_C(p) \cong\B$ otherwise.
\end{enumerate}
 \end{theorem}

In the proof of \Cref{thm:rlsyntacticpairs} we need to take pairs of structures $\A$ and $\B$ such that $\thsigma{n}{\B}\subsetneq \thsigma{n}{\A}$: we do not mention them explicitly in the proof, but we give examples of such structures.

\begin{example}{\cite[Example 2.4]{AK90}}
\label{example:structurehardness}
For any $n \geq 1$, 
\begin{itemize}
\item $\thsigma{2n+1}{\omega^{n}} \subsetneq \thsigma{2n+1}{\omega^{n+1}} $;
  \item $\thsigma{2n+2}{\omega^{n+1}} \subsetneq \thsigma{2n+2}{\omega^{n+1}+\omega^n}$.
\end{itemize}
\end{example}

 \begin{theorem}
 \label{thm:rlsyntacticpairs}
 Let $\{\A,\B\}$ be a pair of structures. For $n \geq 1$,
 \begin{itemize}
	 \item[(i)] if $\{\A,\B\}$ is $\equalitynats^{n+}$-learnable, then  $\thsigma{2n-1}{\A} \neq \thsigma{2n-1}{\B}$; 
	 
	 \item[(ii)] if $\{\A,\B\}$ is $\equalitybaire^{n+}$-learnable, then  $\thsigma{2n}{\A} \neq \thsigma{2n}{\B}$.
	 
 \end{itemize}
 \end{theorem}
 \begin{proof}
 We only prove (ii) as the proof of (i) follows the same pattern.
 
     Let $\Gamma : \ldof{ \{\A,\B\} }  \rightarrow \nats^{\nats^n}$ be a function continuously reducing $\{\A,\B\}$ to $\equalitybaire^{n+}$. Without loss of generality, we can assume that 
	 \[\exists p \in \{[\Gamma(\A)^{[k]}]_{\equalitybaire^{(n-1)+}}: k \in \nats\} \smallsetminus \{[\Gamma(\B)^{[k]}]_{\equalitybaire^{(n-1)+}}: k \in \nats\}\] (otherwise switch the roles of $\A$ and $\B$). Let 
	 \[S_p := \{q \in {\Baire}^{n} : (\exists k)(q^{[k]}\equalitybaire^{(n-1)+} p)\}.\]
 Since $q^{[k]}\equalitybaire^{(n-1)+} p$ is $\Pi_{2n-1}^0$, $S_p$ can be regarded as a $\Sigma_{2n}^{0}$ subset of $\Baire$.

Let $\mathcal{C}$ and $\mathcal{D}$ be structures such that $\thsigma{2n}{\mathcal{D}}\subsetneq \thsigma{2n}{\mathcal{C}}$ (e.g., those provided by \Cref{example:structurehardness}). By \Cref{thm:hardtodistinguish}, we have a continuous operator $\Phi_{S_p}$ such that if $q \in S_p$, then $\Phi_{S_p}(q) \cong \mathcal{C}$ and if $q \notin S_p$, then  $\Phi_{S_p}(q)  \cong \mathcal{D}$. It is easy to check that $\Phi_{S_p} \circ \Gamma$ is a continuous embedding from $\{\A,\B\}$ to $\{\mathcal{C},\mathcal{D}\}$. Finally, by \Cref{lem:relpull} there exists a $\sigmainf{2n}$-sentence that is true in $\A$ and not in $\B$.
 \end{proof}

Combining \Cref{cor:lrdirection} and \Cref{thm:rlsyntacticpairs}, we obtain \Cref{main:thm} restricted to families of size $2$.
\begin{theorem}
\label{thm:pairsofstructures}
	Let $\K$ be a pair of structures. For $n \geq 1$,
	 \begin{itemize}
      \item $\K$ is $=^{n+}_{\nats}$-learnable if and only if $\K$ is a $\sigmainf{2n-1}$-poset;
      \item $\K$ is $=^{n+}_{\nats^\nats}$-learnable if and only if $\K$ is a $\sigmainf{2n}$-poset.
  \end{itemize}
	\end{theorem}

\section{compactness for learning}\label{sec:compactness}
	In this section we prove \Cref{main:thm} for families of structures in general. To do so, we first introduce the following definition.
	\begin{definition}
 \label{def:compactness}
		An equivalence relation $E$ is \emph{compact for learning} if for every family of structures $\K$,  $\K$ is $E$-learnable if and only if for every distinct $\A,\B \in \K$, $\{\A,\B\}$ is $E$-learnable.
	\end{definition}
 Informally, a compact for learning equivalence relation is only affected by the complexity of distinguishing two isomorphic structures, and not by the cardinality of the given family. 
 We highlight an easy but interesting property of compact for learning equivalence relations.
 
 \begin{proposition}
      If $E$ is compact for learning and $F \not\learnreducible{} E$ then there exists a pair of structures which is $F$-learnable and not $E$-learnable.
 \end{proposition}
 By \Cref{thm:pairsofstructures}, to obtain \Cref{main:thm} it suffices to prove the following theorem:

\begin{theorem}
\label{thm:compactness}
	For every $n \geq 1$, $\equalitybaire^{n+}$ and $\equalitynats^{n+}$ are compact for learning.
\end{theorem}
\begin{proof}
We only prove the theorem for $\equalitybaire^{n+}$ as essentially the same proof works for $\equalitynats^{n+}$.

Given a family of structures $\K=\{\A_i:i \in \nats\}$, it suffices to show that, if $\{\A_i,\A_j\}$ is $\equalitybaire^{n+}$-learnable for all $i \neq j$, then  $\K$ is $\equalitybaire^{n+}$-learnable.
By \Cref{thm:pairsofstructures} we have that $\thsigma{2n}{\A_i} \neq \thsigma{2n}{\A_j}$ when $i \neq j$ and we pick a $\sigmainf{2n}$-sentence  $\varphi_{i,j}$ in the symmetric difference $\thsigma{2n}{\A_i} \,\triangle\, \thsigma{2n}{\A_j} $. By \Cref{thm:nicefunction} we have that there is a $\nice{n,\varphi_{i,j}}$ function $\Gamma_{i,j}$. That is, for every $\esse$, if $\esse \models \varphi_{i,j}$ then $\Gamma_{i,j}(\esse)$ is $\full{n}$, and if $\esse \not\models \varphi_{i,j}$ $\Gamma_{i,j}(\esse)$ is $\initial{n}$.
(We are implicitly assuming that $\Gamma_{i,j}=\Gamma_{j,i}$.)
To complete the definition, we set $\Gamma_{i,i} := \Gamma_{0,1}$ for all $i$.
Notice that even when $\esse \notin \ldof{ \{\A_i,\A_j\} }$, $\Gamma_{i,j}(\esse)$ is either $\full{n}$ or $\initial{n}$.
    
We define a continuous reduction $\Gamma$ from $\ldof{\K}$ to $\equalitybaire^{n(\omega+)}$ and, by \Cref{prop:otimesproperty}, we obtain a continuous reduction from $\ldof{\K}$ to $\equalitybaire^{n+}$ as desired. Given $\esse \in \ldof{\K}$, for every $i,j,k$, let \[\Gamma(\esse)^{[\str{\str{i,j},k}]} :=  \distinguish{n,\str{i,j}}{\Gamma_{i,j}^{[k]}(\esse)}.\]

Let $\esse, \esse' \in \ldof{\K}$. We want to show that
$$\esse \cong \esse' \iff \Gamma(\esse)\equalitybaire^{n(\omega+)}\Gamma(\esse').$$
For the left-to-right direction, since $\esse \cong \esse'$ (and hence in particular $\esse$ and $\esse'$ satisfy the same $\sigmainf{2n}$-sentences)  and for any $i\neq j$, $\Gamma_{i,j}$ is $\nice{n,\varphi_{i,j}}$ we have  that $\Gamma_{i,j}(\esse) \equalitybaire^{n(\omega+)} \Gamma_{i,j}(\esse') $. Moreover $\Gamma_{i,i}(\esse) \equalitybaire^{n(\omega+)} \Gamma_{i,i}(\esse') $ for every $i$.  By definition of $\Gamma$ we also have that  $\Gamma(\esse) \equalitybaire^{n(\omega+)} \Gamma(\esse') $.

For the opposite direction, $\esse\not\cong\esse'$. Without loss of generality, assume $\A_i \models \varphi_{i,j}$ and $\A_j\not\models\varphi_{i,j}$. The fact that $\Gamma_{i,j}$ is $\nice{n,\varphi_{i,j}}$ implies that $\Gamma_{i,j}(\esse)$ is $\full{n}$ while $\Gamma_{i,j}(\esse')$ is $\initial{n}$. By this fact and $\Gamma$'s definition, there exists $k$ such that $\Gamma(\esse)^{[\str{\str{i,j},k}]}  = \distinguish{n,\str{i,j}}{p}$ for some $\full{n}$-witness $p$, while there is no $k'$ such that $\Gamma(\esse')^{[\str{\str{i,j},k'}]} = \distinguish{n,\str{i,j}}{q}$ for any $\full{n}$-witness $q$. Hence $\Gamma(\esse) \not\equalitybaire^{n(\omega+)} \Gamma(\esse')$.
\end{proof} 


\section{Conclusions}
\label{sec:finalsection}
In this paper, we study the learning power of the finite Friedman-Stanley jumps of $\equalitynats$ and $\equalitybaire$ proving that these equivalence relations learn natural families of countable structures, namely exactly the families that are pairwise distinguished by $\sigmainf{n}$-sentences for the appropriate $n \in \nats$.

Our study provides an initial exploration of the global properties of the learning hierarchy (i.e., the preorder given by equivalence relations ordered by $\learnreducible{}$). Notably, \Cref{main:thm} implies the existence of a chain of length $\omega$ within this learning hierarchy. Furthermore, the techniques developed to achieve this result can also be applied to evaluate the learning power of Friedman-Stanley jumps of equivalence relations other than $\equalitynats$ and $\equalitybaire$.
In particular, the following result holds:


\begin{theorem}
\label{thm:lowforlearning}
If $E$ is a $\Sigma_2^0$ equivalence relation with uncountably many classes, then $E^+\learnequiv{}\equalitybaire^+$.
\end{theorem}
\begin{proof}
Recall Silver dichotomy (\cite[Theorem 5.3.5]{gao2008invariant}), stating that every coanalytic (and, in particular, Borel) equivalence relation $E$ on a standard Borel space has either countably or perfectly many equivalence classes. Then, by \cite[Proposition 5.1.12]{gao2008invariant}, $\equalitybaire$ continuously reduces to $E$. 
Therefore $\equalitybaire^+$ continuously reduces to $E^+$ and hence ${\equalitybaire^+} \learnreducible{} E^+$.

The proof of ${E^+}\learnreducible{}\equalitybaire^+$ uses the same ideas of \Cref{thm:rlsyntacticpairs}. Indeed, observe that, if $\{\A,\B\}$ is $E^+$-learnable via a continuous reduction $\Gamma$, then there exists some $p \in \{[\Gamma(\A)(k)]_E:k \in \nats\}\smallsetminus\{[\Gamma(\B)(k)]_E: k \in \nats \}$ (otherwise switch the roles of $\A$ and $\B$). Notice that the set $S_p := \{q : (\exists i) (q(i) E p)\}$ is $\Sigma_{2}^{0}$ so that we can apply the same strategy of \Cref{thm:rlsyntacticpairs}. 
\end{proof}

The ideas of the proof of the theorem above can be used to show that, for $n \geq 2$, if $E$ is a $\Sigma_{2n-1}^0$ (resp., $\Sigma_{2n}^0$) equivalence relation such that ${\equalitynats^{(n-1)+}} \learnreducible{} E$ (resp., ${\equalitybaire^{(n-1)+}} \learnreducible{} E$) then $E^{n+} \learnequiv{} \equalitynats^{n+}$ (resp., $E^{n+} \learnequiv{} \equalitybaire^{n+}$).

\medskip

A plethora of questions concerning the learning hierarchy remain open.

	\end{document}